\documentclass[12pt,letterpaper]{article}
\usepackage[margin=1in]{geometry}
\usepackage{graphicx}
\usepackage{enumerate}
\usepackage{amsmath}
\usepackage{amsthm}
\usepackage{amsfonts}

\newcommand{\Z}{\mathbb{Z}}

\newtheorem{thm}{Theorem}[section]
\newtheorem{lemma}[thm]{Lemma}

\newtheorem{cor}[thm]{Corollary}

\if@amsmath
\realias\choose\binom
\else
\renewcommand{\choose}[2]{{{#1}\atopwithdelims(){#2}}}
\fi

\newcommand{\mchoose}[2]{%
  \mathchoice%
    {\left(\kern-0.48em\choose{#1}{#2}\kern-0.48em\right)}
    {\left(\kern-0.30em\choose{\smash{#1}}{\smash{#2}}\kern-0.30em\right)}
    {\left(\kern-0.30em\choose{\smash{#1}}{\smash{#2}}\kern-0.30em\right)}
    {\left(\kern-0.30em\choose{\smash{#1}}{\smash{#2}}\kern-0.30em\right)}
}

\newcommand{\cy}[2]{{{#1}\atopwithdelims[]{#2}}}
\newcommand{\pa}[2]{{{#1}\atopwithdelims\{\}{#2}}}

\newcommand{\ch}{\choose}
\newcommand{\el}{\ell}
\newcommand{\mch}{\mchoose}

\newcommand{\lp}{\left(}
\newcommand{\rp}{\right)}

\newcommand{\noi}{\noindent}

\newcommand{\f}{\frac}

\newcommand{\ha}{\f{1}{2}}

\newcommand{\8}{\infty}

\newcommand{\lbar}[2]{\,\underset{#1}{ \overset{#2}{|} }\, }

\DeclareMathOperator{\Res}{Res}
\DeclareMathOperator{\Real}{Re}

\usepackage{tikz}

\usepackage{verbatim}
\usetikzlibrary{arrows,shapes}
\usetikzlibrary{matrix,decorations.pathmorphing}
\usetikzlibrary{decorations.markings}

\def\PA{preferential arrangement}
\def\BPA{barred \PA}
\def\SBPA{special \BPA}

\def\GKP{Graham, Knuth and Patashnik}
\def\WW{Whittaker and Watson}

\title{Barred Preferential Arrangements}
\author{Connor Ahlbach, Jeremy Usatine and Nicholas Pippenger\footnote{Authors' address: Department of Mathematics, Harvey Mudd College, 301 Platt Blvd., Claremont, CA 91711.
Authors' e-mail addresses: 
$\{${Connor\_Ahlbach}, {Jeremy\_Usatine},{Nicholas\_Pippenger}$\}$@\,{hmc.edu}. }}
\date{\today}

\begin{document}

\begin{titlepage}
\maketitle

\begin{abstract}
A {\em preferential arrangement} of a set is a total ordering of the elements of that set with ties allowed.
A {\em barred} preferential arrangement is one in which the tied blocks of elements are ordered not only amongst themselves but also with respect to one or more bars.
We present various combinatorial identities for $r_{m,\el}$, the number of barred preferential arrangements of $\el$ elements with $m$ bars, using both algebraic and combinatorial arguments.
Our main result is an expression for $r_{m,\el}$ as a linear combination of the $r_k$
($= r_{0,k}$, the number of unbarred preferential arrangements of $k$ elements) for 
$\el\le k\le\el+m$.
We also study those arrangements in which the sections, into which the blocks are segregated by the bars, must be nonempty.
We conclude with an expression of $r_\el$ as an infinite  series that is both convergent and 
asymptotic.
\end{abstract}

\end{titlepage}

\pagenumbering{arabic}

\section{Introduction}

A \emph{preferential arrangement} on $ [\el]  = \{1,\ldots,\el\}$ is a ranking of the elements of $ [\el] $ where ties are allowed. 
For example, the \PA{}s on $ [2] $ include 1 ranked before 2, 2 ranked before 1, and 1 and 2 tied, which we write as
\[
	1,2 \qquad 2,1 \qquad 12
\]
respectively. 
Let $ R(\el) $ denote the set of \PA{}s on $ [\el] $, and
let $ r_{\el} = | R(\el) | $ denote the number of \PA{}s on $ [\el] $. 
For example, from the above, $ r_2 = 3 $. 
We define a \emph{block} of a \PA{} as a maximal set of elements in a \PA{} which are tied in rank. For notation, adjacent numbers represent elements in the same block, and commas separate the blocks. For example, in the \PA{}
\[
	134, 26, 5
\] 
the blocks are $ \{ 1,3,4 \}, \{ 2,6 \}, $ and $ \{ 5\} $.


A \emph{barred \PA} on $ [\el] $ with $ m $ bars is a ranking of the elements of $ [\el] $ where ties are allowed, and $ m $ bars are placed to separate the blocks into $ m +1 $ \emph{sections}. 
No bar can divide a block in two. 
Section $ 0 $ is the region before the first (leftmost) bar. 
Section $  m $ is the region after the last (rightmost) bar. 
And, for all $ 1 \le i \le m - 1 $, section $ i $ is the region between the $i$th and $ (i + 1) $th bars from the left. 
Each section is its own \PA. 
For example, the \BPA{}s on $ [1] $ with 2 bars are
\[
	1 | | \qquad | 1 | \qquad | | 1.
\] 
The \BPA
\[
	183, 4 | 56,7 | 92
\]
is a \BPA{} on $ [9] $ with 2 bars where section 0 is $ 183, 4 $, section 1 is $ 56, 7 $ and section 2 is $ 92 $. 
Let $ R(m, \el) $ denote the set of \BPA{}s on $ [\el] $ with $ m $ bars, and let $ r_{m,\el} = | R(m,\el) | $ denote the number of \BPA{}s on $ [\el] $ with $ m $ bars. 
For example, from the above, $ r_{2,1} = 3 $. 


We provide a table of values of $ r_{m,\el} $ for small values of $ m $ and $ \el $:

\begin{center}
\begin{tabular}{|*{10}{*{1}{l}|}}
\hline
$m\backslash \ell$ & ${\bf0}$ & ${\bf1}$ & ${\bf2}$ & ${\bf3}$ & ${\bf4}$ & ${\bf5}$ & ${\bf6}$ & ${\bf7}$ & ${\bf8}$\\
\hline
${\bf0}$& $1$ & $1$ & $3$ & $13$ & $75$ & $541$ & $4683$ & $47293$ & $545835$ \\
${\bf1}$& $1$ & $2$ & $8$ & $44$ & $308$ & $2612$ & $25988$ & $296564$ & $3816548$ \\
${\bf2}$& $1$ & $3$ & $15$ & $99$ & $807$ & $7803$ & $87135$ & $1102419$ & $15575127$ \\
${\bf3}$& $1$ & $4$ & $24$ & $184$ & $1704$ & $18424$ & $227304$ & $3147064$ & $48278184$ \\
${\bf4}$& $1$ & $5$ & $35$ & $305$ & $3155$ & $37625$ & $507035$ & $7608305$ & $125687555$ \\
${\bf5}$& $1$ & $6$ & $48$ & $468$ & $5340$ & $69516$ & $1014348$ & $16372908$ & $289366860$ \\
\hline
\end{tabular}
\end{center}
 
The notion of a preferential arrangement occurs 
if $\el$ candidates have been interviewed and evaluated for a position; a preferential arrangement of $\el$ elements may be used to indicate the order (with possible ties) of their suitability for the position.
The term ``preferential arrangement'' seems to be due to Gross \cite{gross:pas}  in 1962, though the concept had been described in a paper by Touchard \cite{touch:pas} in 1933.
The numbers $r_\el$  appeared even earlier in connection with a problem concerning trees in a paper by Cayley \cite{cayley:trees} in 1859.
Barred preferential arrangements with a single bar were introduced by Pippenger \cite{pipp:pas}, who showed that 
\begin{equation}
\label{eq:1.1} 
	r_{1,\el} = {1\over 2}r_\el + {1\over 2}r_{\el+1}. 
\end{equation}
If $\el$ candidates have been interviewed  for a position, a single bar might be used to separate the candidates who are worthy of being hired from those who are not (with distinctions being possible among the unworthy as to their degree of unworthiness).

Our goal in this paper is to study the case of multiple bars.
If, in the situation involving $\el$ candidates, there are $m$ ranks into which candidates may be hired, the first $m-1$ bars might be used to separate the candidates who are suitable for the various ranks (assuming of course that a candidate who is suitable for a given rank is automatically suitable for all lower ranks).
In Section 2 we shall generalize (\ref{eq:1.1}) to the case of multiple bars. 
Our main result is
\[
	r_{m, \el} = \f{ 1}{ 2^m m!} \sum_{ i = 0 }^{m} \cy{m + 1}{ i + 1} r_{ \el + i},
\]
where $\cy{n}{ k}$ is the Stirling number of the first kind, the number of permutations of $n$ elements having $k$ cycles (see \GKP{} \cite{graham:concrete}, Section 6.1).
We shall give both algebraic and combinatorial (that is, bijective) proofs.
In Section 3 we shall derive a number of other identities involving the $r_{m,\el}$.
In Section 4 we shall explore a variant of barred preferential arrangements for which the sections, into which the blocks are segregated by the bars, are required to be nonempty.
Finally, in Section 5 we shall extend the known asymptotic results concerning $r_\el$, obtaining an infinite series that is at once
both asymptotic and convergent.


\section{Enumerating Barred Preferential Arrangements}

In this section, we shall express $r_{m,\el}$ as a linear combination of the $r_k$ for 
$\el\le k\le\el+m$.
We begin by generalizing (\ref{eq:1.1}) in Theorem  \ref{Recurrence}, which expresses $r_{m,\el}$ in terms of  $r_{m-1,\el}$ and $r_{m-1,\el+1}$, and which we prove by constructing an explicit bijection.
Our main result then appears as Theorem \ref{Main}, for which we give two proofs, the first by induction using Theorem \ref{Recurrence}, and the second by again constructing an explicit bijection.


\begin{thm}
\label{Recurrence}
For $m\ge 1$, we can write $ r_{m,\el} $ in terms of the previous sequence $ \{ r_{m -1, k} \} $ as
\[
	r_{m, \el} = \f{ 1}{ 2m } r_{ m - 1, \el + 1} + \ha r_{ m - 1, \el }.
\]
\end{thm}

\begin{proof}

We prove this result combinatorially by establishing a bijection
\[
	f: \{ 0, 1 \} \times [m] \times R(m,\el) \to R(m - 1, \el + 1) \cup \lp R( m - 1, \el) \times ( 0 \cup [m - 1] ) \rp.
\]
Here, $ [m] $ chooses 1 bar out of the $ m $ bars. 
Then, $ \{ 0, 1 \} $ labels this bar with a binary label, which is either 0 or 1.  
Thus, $ \{ 0, 1 \} \times [m] \times R(m, \el) $ represents the set of all \BPA{}s with $ m $ bars, where one bar is given a binary label. 
Consider any $ X \in R(m,\el) $, where 1 bar has a binary label. 
Let $ B $ be the bar with the binary label. 
Then $ f $ acts on $ X $ as follows: 
\begin{itemize}

\item If $ B $'s binary label is 0, replace $ B $ with $ (\el + 1) $ in its own block. 
For example,

	$$ 123 \lbar{0}{} \mapsto 123, 4. $$

\item If $ B$'s binary label is 1 and there is a block directly to the left of $B$, remove $ B $ and adjoin $ (\el + 1) $ to that block. 
For example,

	$$ 123 \lbar{1}{} \mapsto 1234. $$

\item If $ B $'s binary label is 1 and there is not a block directly to the left of $B$, remove $ B $ to get a \BPA{} $ A $ with $ m - 1 $ bars. 
Then, either $ B $ was on the left end or $ B $ was directly to the right of a bar.
If $ B $ was on the left end, set $ f_0 (X) = 0 $. 
If $ B $ was directly after the $ k $-th bar from the left, set $ f_0(X) = k $. 
Define $ f(X) = (A, f_0(X)) $. 
For example,   

	$$ 123 \lbar{}{} \lbar{1}{} \mapsto ( 1 \lbar{}{}, (1) ). $$

\end{itemize}

We next show that we can invert $ f $. 
Suppose we are given $ Y \in R(m - 1, \el + 1) $. 
For $ f $ to map to $ Y $, $ f $ must have added $ (\el + 1) $. 
Thus, we first find $ (\el + 1) $ in $ Y $. 
If $ \el + 1 $ is in its own block, we replace it with a bar with binary label 0. 
By the definition of $ f $, this is the only \BPA{} that could and does map to $ Y $. 
If $ (\el + 1) $ belongs to a block with other elements, we remove it and place a bar with binary label 1 just to the right of this block. 
By the definition of $ f $, this is the only \BPA{} that could and does map to $ Y $. 
Hence, for each $ Y \in R(m - 1, \el + 1) $, there exists a unique $ X \in R(m, \el) $ with a binary label such that $ f(X) = Y $.

Now, suppose we are given $ (Z, a) \in R(m - 1, \el) \times ( 0 \cup [m - 1] ) $. 
For $ f $ to map to $ Z $, we must have removed a bar without inserting $\el+1$. 
If $ a = 0 $, place a bar on the left end of $ Z $ with binary label 1. 
If $ a \ne 0 $, place a bar just to the right of the $ a $-th bar from the left, with binary label 1. 
By the definition of $ f $, this is the only \BPA{} that could and does map to $ (Z,a) $. 
Hence, for each $ (Z, a) \in R(m - 1, \el) \times ( 0 \cup [m - 1] ) $, there exists a unique $ X \in R(m, \el) $ with a binary label such that $ f(X) = (Z, a) $. 
Thus $ f $ is invertible.

Since $f$ is a bijection, we conclude that 
\[
	| \{ 0, 1 \} \times [m] \times R(m,\el) | 
	= | R(m - 1, \el + 1) \cup \lp R( m - 1, \el) \times ( 0 \cup [m - 1] ) \rp |.
\]
Since the first union on the right-hand side is disjoint, we have
$$2m r_{m,\el} = r_{ m - 1, \el + 1} + m r_{ m -1, \el },$$
which completes the proof.
\end{proof}


\begin{thm}
\label{Main}
For $m\ge 1$, we can write $ r_{m,\el} $ in terms of the original sequence $ \{ r_{\el} \} $ as
\[
	r_{m, \el} = \f{ 1}{ 2^m m!} \sum_{ i = 0 }^{m} \cy{m + 1}{ i + 1} r_{ \el + i}. 
\]
\end{thm}

\begin{proof}
(Method 1: Induction using Theorem \ref{Recurrence}.) \\
\emph{Base Case:} For $ m = 0 $,
\[
	\f{ 1}{ 2^m m!} \sum_{ i = 0 }^{m} \cy{m + 1}{ i + 1} r_{ \el + i} = \f{ 1}{ 2^0 0!} \sum_{ i = 0 }^{0} \cy{1}{ i + 1} r_{ \el + i} = \cy{1}{1} r_\el = r_{0,\el}, 
\]
proving the result for $ m = 0 $. 
Now suppose $ m \ge 1 $.

\noi \emph{Inductive Hypothesis:} Assume the result holds for $ m - 1 $. 
That is,
\[
	r_{m - 1, k} = \f{ 1}{ 2^{m - 1} (m - 1)!} \sum_{ i = 0 }^{m - 1} \cy{m}{ i + 1} r_{k + i} 
\]
for all $ k \ge 0 $. From Theorem \ref{Recurrence}, 
\begin{align*}
	r_{ m, \el} & = \f{ 1}{ 2m} r_{ m - 1, \el + 1} + \ha r_{ m - 1, \el} \\
	    & = \f{ 1}{ 2m} \lp  \f{ 1}{ 2^{m - 1} (m - 1)!} \sum_{ i = 0 }^{m - 1} \cy{m}{ i + 1} r_{ \el + 1+ i} \rp + \ha \lp     \f{ 1}{ 2^{m - 1} (m - 1)!} \sum_{ i = 0 }^{m - 1} \cy{m}{ i + 1} r_{ \el + i}  \rp \\
	&= \f{1}{ 2^m m!} \lp \sum_{ j = 1}^m \cy{m}{j} r_{ \el + j} +  m \sum_{ j = 0}^{m - 1} \cy{m}{j + 1} r_{ \el + j} \rp. 
\end{align*}
Noticing that $ \cy{m}{0} = \cy{m}{m + 1} = 0 $ and combining the sums, 
\[
	r_{m, \el} = \f{1}{ 2^m m!} \lp \sum_{ j = 0}^{m} \lp \cy{m}{j} + m \cy{m}{j + 1} \rp r_{ \el + j} \rp.
\]
For the unsigned Stirling numbers of the first kind, we have $ \cy{m + 1}{j + 1} = \cy{m}{ j } +  m \cy{m}{j + 1} $ (see \GKP{} \cite{graham:concrete}, p.~250). 
Hence,
\[
	r_{ m, \el} = \f{1}{ 2^m m!} \lp \sum_{ j = 0}^{m} \cy{m + 1}{j + 1} r_{ \el + j} \rp.
\] \end{proof}

\begin{proof}

(Method 2: Bijective combinatorial proof.) \\

\noi We can iterate the map used to prove Theorem \ref{Recurrence} to establish this more general result. 
Let $ S_m $ denote the set of permutations of $ [m] $. 
Let $ C(n,k) $ denote the set of permutations of $ [n] $ with $ k $ cycles. 

We prove this result by establishing a bijection
\[
	g: \{ 0, 1 \}^m \times S_m \times R(m,\el) \to \bigcup_{ i = 0}^m C(m + 1, i + 1) \times R(\el + i).
\]
Here, $ \{ 0, 1 \}^m $ gives each of the $ m $ bars a binary label, 0 or 1. 
Also, $ S_m $ gives each of the $ m $ bars a distinct order label from $ [m] $. 
Thus, $ \{ 0, 1 \}^m \times S_m \times R(m,\el) $ represents the set of all BPAs with $ m $ bars where each bar has a binary and order label. 
When we refer to bar $ x $, we mean the bar with order label $ x $. 
Consider any \BPA{} $ X \in R(m,\el) $ with order and binary labels. 
Then, $ g $ acts on each bar in the order of increasing order labels just as before:
\begin{itemize} 

\item If its binary label is 0, replace it with the next integer not yet used in the \BPA{} in its own block.

\item If its binary label is 1 and there is a block directly left of the bar, remove it and adjoin the next integer not yet used in the \BPA{} in that  block.

\item If its binary label is 1 and there is not a block directly left of the bar, remove the bar.

\end{itemize}
After $ g $ has acted on all of the bars we end up with a \PA{} we shall call $ g_{PA}(X) $. 
Then, as we can add 0 or 1 elements for each of $ m $ bars, $ g_{PA}(X) \in R(0, \el + i) $ for some integer $ i, 0 \le i \le m $. But, $ g $ also yields a permutation $ g_C(X) $ of $ [m + 1] $, constructed as follows:
\begin{itemize}

\item Place an extra bar with order label $ (m + 1) $ at the left end of the \BPA. 
This extra bar will effectively act as the left end of the \BPA, and it will make our proof a bit more straightforward.

\item We define a third label, the cycle label, on the bars. 
Let $ c(x) $ denote the cycle label of bar $ x $. We initialize the cycle labels as the order label: $ c(x)= x $.

\item Whenever bar $ a $ is removed, there must have been a bar $ b $ directly left of it ($ b = m + 1 $ if $ a $ was at left end). 
After removing bar $ a $, we append the cycle label of $ a $ onto the end of that of bar $ b $:
\[
	c(b) \longleftarrow ( c(b) \, c(a) ).
\]
\item The first element of $ c(x) $ is always $ x $ because we always append to the end. 
Also, $ x $ is the maximum element of $ c(x) $ because all elements merged into $ c(x) $ must have been acted on by $ g $ before $ x $ and so must be less than $ x $. 

\item Whenever a bar, say bar $ y $, is replaced with the next number not used, either in its own block  or in the block directly to its left, make its cycle label $ c(y) $ a cycle in the permutation $ g_C(X) $. 

\item After all $ m $ bars are removed, remove the extra bar $ m + 1 $, and make its cycle label $ c(m + 1) $ a cycle in the permutation $ g_C(X) $. 

\end{itemize} 
Finally, we define 
\[
	g(X) = ( g_C(X), g_{PA}(X) ). 
\]
For the example below, we write the labels of the bars as follows:
\[
	\lbar{\text{ Binary \, Label }}{ \text{ Order Label (Cycle Label)} }	
\]
\begin{align*}
X = &\lbar{}{5(5)} \lbar{0}{3(3)} 1 2 \lbar{1}{4(4)} \lbar{1}{1(1)} \lbar{1}{2(2)} 3\\
\mapsto & \lbar{}{5(5)} \lbar{0}{3(3)} 1 2 \lbar{1}{4(41)}  \lbar{1}{2(2)} 3\\
\mapsto &\lbar{}{5(5)} \lbar{0}{3(3)} 1 2 \lbar{1}{4(412)} 3 \\
\mapsto &(3)\lbar{}{5(5)} 4, 1 2 \lbar{1}{4(412)} 3\\
\mapsto &(412)(3) \lbar{}{5(5)} 4, 1 2 5 , 3\\
\mapsto &(5)(412)(3) \, 4, 1 2 5,  3
\end{align*}
$$ g_C(X) = (5)(412)(3), \qquad g_{PA}(X) = 4, 1 2 5, 3. $$
Now, every time we substitute another number, we add one cycle to the permutation. 
We have 1 more cycle from extra bar $ m + 1 $. 
Hence, $ g_{PA}(X) \in R(\el + i) $ if and only if $ g_C(X) $ has $ i + 1 $ cycles, or $ g_C(X) \in C(m + 1, i + 1) $. 
Thus, as claimed, $ g $ is a map
\[
	g: \{ 0, 1 \}^m \times S_m \times R(m,\el) \to \bigcup_{ i = 0}^m C(m + 1, i + 1) \times R(\el + i).
\]

Next, we show that we can invert $ g$. 
Given permutation $ Y \in C(m + 1, i + 1) $ and \PA{} $ Z \in R(\el + i ) $ for $ 0 \le i \le m $, we find an $ X $ with its order and binary labels such that 
\[
	g(X) = (Y, Z).
\] 
We reconstruct such an $ X $ and show it is unique. 
We can add back the bars with their order and binary labels using the information contained in $ Y $. 
First, we write $ Y $ in terms of its cycles, with each cycle starting at its maximum. 
Also, we know that the largest $ i $ elements of $ Z $ must have been added by $ g $. 

By construction, each cycle represents a sequence of bar removals that each terminate in the addition of a new integer to the \BPA. 
By definition, a new integer does not replace bar $ x $  if and only if $ c(x) $ was appended to the end of another cycle label at some point. 
Thus, any bar that remains at the start of a cycle must have been replaced by the next integer not yet used. 
Furthermore, any bar not at the start of a cycle must have been removed without replacement of the next integer not yet used. 
Because the steps in the definition of $ g $ are ordered by increasing order labels, we know the order in which the cycles were created---in increasing order of their maxima. 
We also know the order in which the integers were added---increasing order. 
Note that the cycle containing $ (m + 1) $ corresponds to the extra bar labeled $ (m + 1) $ on the left end.  

Thus, by comparing the orders of the maxima in the cycles and the new integers created, we can uniquely determine which cycles correspond to which added integers. 
(The cycle containing $ (m + 1) $ does not correspond to an integer, but to the the left end of the \BPA.) 
The cycle and the corresponding integer are created simultaneously. 
If added integer $ y $ corresponds to cycle $ C  = (c_1 \, c_2 \, \cdots \, c_k) $. 
Then, we must have had the sequence of adjacent bars in this order: $ c_1, c_2, \cdots c_k $ just right of or at where $ y $ was added. 
To illustrate,
\[ 
	( \text{ Where $ y $ was added } ) \lbar{}{c_1} \lbar{1}{c_2} \lbar{1}{c_3}\dots \lbar{1}{c_k}.
\]
This holds because appending cycle labels preserves the order from left to right, and this cycle must have been created when $ y $ was added. 
Because bars $ c_2, c_3, \cdots c_k $ were removed without replacement, their binary labels must have been 1. We know $ c_1 $ was substituted with $ y $. 
So, by the definition of $ g $,
\begin{itemize}

\item If $ y $ is in its own block, bar $ c_1 $'s binary label must have been 0.

\item If $ y $ is in a block with other integers, bar $ c_1 $'s binary label must have been 1.

\end{itemize}  
Let the cycle containing $ (m +1) $ be $ ( (m + 1) \, d_1 \, d_2 \, \cdots \, d_k ) $. 
Then we must have had the sequence of adjacent bars:
\[
	(\text{ Left End } ) ( \lbar{}{m + 1} ) \lbar{1}{d_{1}} \lbar{1}{d_{2}} \cdots \lbar{1}{d_{k}}.
\]
We remove $ m + 1 $ before returning the final \BPA, $ X $.

We assume that the blocks are increasing. 
For the case of multiple integers added in a single block, we must order the cycles by increasing order of their maxima from left to right. 
Suppose that $ \{ n_i \}_{ i = 1}^r $, where the $ n_i $ increase with $ i $, are added in a block and have the corresponding cycles $ \{ C_i \}_{ i = 1}^r $, respectively. Let
\[
	C_i = ( c_{ i,1} \, c_{ i , 2} \, \cdots \, c_{ i, k_i} ).
\]
Then we must have
\[
	 ( \text{ Where $ \{ n_i \}_{ i = 1}^r $ were added } ) \lbar{}{c_{1,1}} \lbar{1}{c_{1,2}}\dots \lbar{1}{c_{1,k_1}} \lbar{}{c_{2,1}} \lbar{1}{c_{2,2}}\dots \lbar{1}{c_{2,k_2}} \dots \lbar{}{c_{m,1}} \lbar{1}{c_{m,2}}\dots \lbar{1}{c_{m,k_m}}.
\]
This must happen so that the bars in $ C_1, C_2, \cdots C_{i - 1} $ are removed before $ C_i $ and $ n_i $ are created. 
Let $ X $ be the \BPA{} constructed by this reversal. For example, suppose that we are given
\[
	Y = (83)(714)(65)(2), \qquad Z = 7, 123, 456.
\]
First, $ Y $ has 4 cycles, so 5, 6, and 7 were added. 
Then, $ (83) $ corresponds to the extra bar on the left end, $ (714) $ corresponds to 7, $ (65) $ corresponds to 6, and $ (2) $ corresponds to 5. 
With the extra bar, this must have came from
\[
	\lbar{}{8} \lbar{1}{3} \lbar{0}{7}  \lbar{1}{1}  \lbar{1}{4} 123, 4  \lbar{1}{2}  \lbar{1}{6}  \lbar{1}{5}.
\]
Removing the extra bar, we have
\[
	X = \lbar{1}{3} \lbar{0}{7}  \lbar{1}{1}  \lbar{1}{4} 123, 4  \lbar{1}{2}  \lbar{1}{6}  \lbar{1}{5}.
\]
We replaced the largest $ i $ elements of $ Z \in R(\el + i) $ with a sequence of bars, so $ X $ is a labeled BPA on $ [\el] $. 
Also, we added back one bar for each element of the domain of permutation $ Y $, except $ (m + 1) $, which corresponds to the extra bar. 
Thus, $ X $ has $ m $ bars, so $ X \in R(m, \el) $ with binary and order labels. 
Then, 
\[
	g(X) = (Y, Z)
\] 
because the sequences of adjacent bars added back for each cycle add the desired integers in $ Z $ and the desired cycles in $ Y $. 
And, such an $ X $ is unique because, as has been argued, the sequences of adjacent bars with their binary and order labels and their placement in the \BPA{} are unique. 
Hence, this $ X $ is unique. 
Since $ g $ is a bijection, we conclude that 
\[
	| \{ 0, 1 \}^m \times S_m \times R(m,\el) | 
	= \left\vert \bigcup_{ i = 0}^m C(m + 1, i + 1) \times R(\el + i) \right\vert. 
\]
Since the union on the right-hand side is disjoint, we have
$$2^m m! \,r_{m, \el} = \sum_{ i = 0 }^m \cy{m + 1}{ i + 1} r_{ \el + i },$$	
which completes the proof.
\end{proof}


\section{ Identities for Barred Preferential Arrangements }

We begin with a formula expressing $r_{m,\el}$ as a sum.
The formula
$$r_\el = \sum_{k=0}^\el \pa{\el}{k}\,k!$$
is implicit in the work of Touchard \cite{touch:pas}, and the formula
$$r_{1,\el} = \sum_{k=0}^\el \pa{\el}{k}\,(k+1)!$$
was established by Pippenger \cite{pipp:pas}.
We generalze these formulas as follows.

\begin{thm} 
\label{Stir1}
For $m\ge 0$ and $\el\ge 1$, we have
\[
	r_{ m,\el} = \sum_{ k = 0}^\el \pa{\el}{k} k! \mch{m + 1}{k},
\]
where $\pa{\el}{k}$ is the Stirling number of the second kind,
the number of partitions of $\el$ elements into $k$ blocks
(see \GKP{} \cite{graham:concrete}, Section 6.1),
and $\mch{n}{k} = \ch{n+k-1}{k}$ is the number of ways of choosing $k$ elements 
to form a multiset (repetitions are allowed, with multiplicities summing to $k$) from a set of $n$ distinct elements.
\end{thm}

\begin{proof}
Suppose that our \BPA{} has $ k $ blocks. 
First, we can partition $ [\el] $ into $ k $ unordered blocks in $ \pa{\el}{k} $ ways. 
Then, we can order these blocks in $ k! $ ways. 
Finally, we have $ (k + 1) $ positions before, between and after these blocks in which to place the $m$ bars, and each position can have zero or more  bars. 
Hence, we can place the $m$ bars in 
\[
	\mch{k + 1}{m} = \ch{m + k}{m} = \ch{m + k}{k} = \mch{ m + 1}{ k}
\] 
ways. 
Thus the number of \BPA{}s on $[\el]$ with $m$ bars and $k$ blocks is
\[
	\pa{\el}{k} k! \mch{m + 1}{k}.
\]
Summing over $ k $ completes the proof.
\end{proof}


Next we turn to the exponential generating function
$$r_m(z) = \sum_{\el\ge 0} {r_{m,\el} \, z^\el \over \el!}.$$
Cayley \cite{cayley:trees} derived the case $m=0$,
\begin{equation}
\label{eq:cayleygen}
	r(z) = {1\over 2-e^z},
\end{equation}
and Pippenger \cite{pipp:pas} gave the result for $m=1$,
$$r_1(z) = {1\over (2-e^z)^2},$$
We generalize these results as follows.

\begin{thm} 
\label{EGF2}
For $m\ge 0$, we have
\[
	r_m(z) = \sum_{ \el \ge 0} {r_{m, \el} \, z^\el \over \el!} = {1\over ( 2 - e^z )^{m + 1}}. 
\]
\end{thm}


\begin{proof}
We can construct a \BPA{} on $[\el]$ with $m\ge 1$ bars and $k\ge 0$ elements before the first bar 
by (1) selecting the $k$ elements that appear before the first bar (this can be done in $\ch{\el}{k}$ ways), then (2) arranging these $k$ elements  in a \PA{} (this can be done in $r_k$ ways), and finally (3) arranging the remaining $\el-k$ elements in a \BPA{} with $m-1$ bars.
Summing over $k$ yields
$$r_{m,\el} = \sum_{k=0}^\el \ch{\el}{k} \, r_k \, r_{m-1,\el-k}.$$
If $u_\el$ and $v_\el$ are sequences with exponential generating functions
$u(z)$ and $v(z)$, respectively, then the sequence 
$w_\el = \sum_{0\le k\le \el} \choose{\el}{ k} \, u_k \, v_{\el-k}$ obtained from them by 
``binomial convolution'' has the exponential generating function $w(z) = u(z) \, v(z)$
(see \GKP{} \cite{graham:concrete}, p.~351, (7.74)).
Thus we obtain 
$$r_m(z) = r_0(z) \, r_{m-1}(z).$$
The theorem now follows from (\ref{eq:cayleygen}) by induction on $m$.
\end{proof}

We can use this generating function to provide another proof of Theorem \ref{Recurrence}.

\begin{proof}
By definition,
\[
	r_{ m - 1, \el + 1} + m r_{ m - 1, \el} 
	= \left[ \f{ z^{ (\el + 1) } }{ (\el + 1)!} \right] r_{m - 1}(z) 
	+ m \left[ \f{ z^{ \el } }{ \el!} \right] r_{m - 1}(z),
\]
where $[z^\el/\el!]$ denotes $\el!$ times the coefficient of $z^\el$ in what follows.
But differentiation of an exponential generating function shifts the sequence it generates down by 1. 
Thus
\[
	r_{ m - 1, \el + 1} + m r_{ m - 1, \el} = \left[ \f{ z^{ \el } }{ \el!} \right] r_{ m - 1}'(z) +  \left[ \f{ z^{ \el } }{ \el!} \right] m r_{m - 1}(z) =  \left[ \f{ z^{ \el } }{ \el!} \right] r_{ m - 1}'(z) + m r_{m - 1}(z).
\]
Since $ r_{m - 1}(z) = ( 2 - e^z )^{- m} $, we have
\begin{align*}
	r_{ m - 1, \el + 1} + m r_{ m - 1, \el} & = \left[ \f{ z^{ \el } }{ \el!} \right] \f{ m e^z}{ (2 - e^z)^{m + 1} } + \f{ m}{ (2 - e^z)^m } = m \left[ \f{ z^{ \el } }{ \el!} \right] \f{ e^z + ( 2 - e^z) }{  (2 - e^z)^{m + 1} } \\
	& = 2m  \left[ \f{ z^{ \el } }{ \el!} \right] \f{ 1 }{ (2 - e^z)^{m + 1} } = 2m \left[ \f{ z^{ \el } }{ \el!} \right] r_m(z) = 2m r_{ m, \el.}
\end{align*}
Thus,
\[
	2m r_{m,\el} = r_{ m - 1, \el + 1} + m r_{ m - 1, \el},
\]
which completes the proof. 
\end{proof}


\section{ Special Barred Preferential Arrangements }

In a \BPA, sections can be empty. 
What happens if we exclude those \BPA{}s with empty sections? 
How many will be left? 
We define a {\em\SBPA} to be a \BPA{} with no empty sections.
For example, the \BPA{}
\[
	14,3 | 26 | 7
\] 
is special, but the \BPA{}s
\[
	| 14,3 | 26 | 7 \qquad 14,3 | 26 | | 7
\]
are not, since sections 0 and 2, respectively, are empty. 
Let $ S(m,\el) $ be the set of \SBPA{}s on $ [\el] $ with $ m $ bars, and
let $ s_{ m,\el} = | S(m,\el) | $ be the number of such \SBPA{}s. 
If  $ \el = 0 $, we have no elements in $[\el]$, so at least one section is empty. 
Thus, for $ m \ge 0$, we have $s_{ m, 0} = 0, $ as opposed to $ r_{m,0} = 1 $. 

In this section,
we will prove various identities for $ s_{m,\el} $.
We begin with a formula expressing $s_{m,\el}$ as a sum. 

\begin{thm} 
For $m\ge 0$ and $\el\ge 1$, we have
\[
	s_{ m,\el} = \sum_{ k = 1}^\el \pa{\el}{k} k! \ch{k - 1}{m}. 
\]
\end{thm}

\begin{proof}
Suppose that our \SBPA{} has $ k $ blocks. 
First we can partition $ [\el] $ into $ k $ unordered blocks in $ \pa{\el}{k} $ ways.
 Then, we can order these blocks in $ k! $ ways. 
 Finally, we have $ (k - 1) $ positions between these blocks to place the $m$ bars, and each position can have at most 1 bar, because we can have no empty sections. 
 Thus, we can place our bars in $ \ch{ k - 1}{m} $ ways. 
 Since $1\le k\le m$, summing over $ k $ completes the proof.
\end{proof}


Next we turn to the exponential generating function
$$s_m(z) = \sum_{\el\ge 0} {s_{m,\el} \, z^\el \over \el!}.$$
We begin with the case $m=0$.

\begin{lemma}
\label{EGFs0}
We have
\[
	s_0(z)  = r(z) - 1 = \f{e^z- 1}{ 2 - e^z }.
\]
\end{lemma}

\begin{proof}
If there are no bars, there is just one section, and this section will be empty if and only if $\el=0$.
Thus, $s_0(z)$ is obtained by omitting the constant term $r_{0,0}\cdot z^0 = 1$ from $r_0(z)$.
\end{proof}

\begin{lemma} 
\label{Convs}
For $ m \ge 1 $, we have
\[
	s_{m,\ell} = \sum_{k=0}^{\ell} \choose{\ell}{k}s_{m-1,k}s_{0,\ell-k}.
\]
\end{lemma}

\begin{proof}
Consider a \SBPA{} of $ [\el] $ with $ m $ bars. 
Suppose there are $ k $ elements to the left of the rightmost bar. 
We first choose $k$ elements of $[\ell]$ to be left of the rightmost bar in $ \ch{\ell}{k} $ ways. 
We then arrange these elements with $ m - 1 $ bars into a \SBPA{} in $ s_{m - 1, k} $ ways. 
The final bar is placed right of this arrangement, and to the right of that we preferentially arrange the remaining $\ell - k$ elements with no bars in $ s_{0, \ell- k} $ ways. 
Summing over $ k $ completes the proof.
(If $m\ge 1$, then $k$ must satisfy $1\le k\le \el-1$, but the terms 
in the sum corresponding to $k=0$ and $k=\el$ vanish.)
\end{proof}


\begin{thm}
\label{EGFs}
For $m\ge 1$, we have
\[
	s_m(z) =  (r(z) - 1)^{ m + 1} = \left(\frac{e^z - 1}{2-e^z}\right)^{m+1}.
\]
\end{thm}

\begin{proof}
By Lemma \ref{Convs}, $s_{m,\el}$ is obtained from $s_{m-1,k}$ and $s_{0,j}$ by binomial convolution, so we have $ s_m(z) = s_{ m - 1}(z) s_0(z)$. 
The theorem now follows from Lemma \ref{EGFs0} by induction on $m$.
\end{proof}


Our next result expresses $s_{m,\el}$ in terms of $r_{i,\el}$ for $0\le i\le m$.

\begin{thm}
\label{foo}
For $m\ge 0$ and $ \el \ge 1 $, we have
\[
	s_{m, \el} = \sum_{ i = 0}^m (-1)^{ m - i} \ch{m + 1}{  i + 1 } r_{i, \el}.
\]
\end{thm}

\begin{proof} 
From Theorem \ref{EGFs},  
\begin{align*}
	s_m(z) & = (r(z) - 1)^{m + 1} = \sum_{ k = 0}^{ m + 1} \ch{m + 1}{k} (-1)^{ m + 1 - k } r(z)^k\\
            & = \sum_{ i = 0 }^m \ch{ m + 1}{ i + 1 } (-1)^{ m - i } r(z)^{ i + 1} + (-1)^{ m + 1}.
\end{align*}
Therefore, for all $ \el \ge 1 $,
\begin{align*}
	s_{m,\el} & = \left[ \f{z^\el}{\el!} \right] s_m(z) = \sum_{ i = 0 }^m \ch{ m + 1}{ i + 1 } (-1)^{ m - i } \left[ \f{z^\el}{\el!} \right] r(z)^{ i + 1} + \left[ \f{z^\el}{\el!} \right] (-1)^{ m + 1} \\
	& = \sum_{ i = 0 }^m (-1)^{ m - i } \ch{ m + 1}{ i + 1 } r_{ i, \el }.
\end{align*}
 \end{proof}
 
 We can use the principle of inclusion-exclusion to provide another proof of this theorem.
Let $ A_j $ denote the set of barred preferential arrangements in which section $ j $ is empty. 
Now, forming a \BPA{} in which all the sections $ j_1, j_2 \cdots j_i $ are empty is preferentially arranging $ [\el] $ among the other $ m + 1 - i $ sections, or equivalently a \BPA{} with $ m - i $ bars. 
Therefore,
\[
\left\vert \bigcap_{ k = 1}^i A_{ j_k} \right\vert = r_{ m - i, \el}
\]
It follows by the principle of inclusion-exclusion that the number of \SBPA{}s with $ m $ bars on $ [\el] $ is
\[
s_{m, \el} = \sum_{ i = 0}^m (-1)^{ i} \ch{m + 1}{ i } r_{m - i, \el} = \sum_{ i = 0}^m (-1)^{ m - i} \ch{m + 1}{ i + 1 } r_{i, \el}.
\]


The following theorem expresses $s_{m,\el}$ in terms of $r_{m,\el-j}$ for $0\le j\le \el$.

\begin{thm} For $m\ge 0$ and $ \el \ge 0 $, we have
\[
	s_{m,\ell} = (m+1)! \sum_{j=0}^\ell \choose{\ell}{j}\pa{j}{m+1}r_{m, \ell-j}.
\]
\end{thm}

\begin{proof} 
Consider a \SBPA{} of $ [\el] $ with $ m $ bars. 
Suppose there are $ j $ elements in the first blocks of all sections. First, we choose $j$ elements from $[\ell]$ that will be in first blocks of sections in $ \ch{\ell}{j} $ ways. 
We then partition these $j$ elements into $m+1$ blocks, in $ \pa{j}{m + 1} $ ways. 
We order these blocks to assign them to $m+1$ sections in $ (m + 1)! $ ways. 
Now we know that each section has at least one block and is thus not empty. 
The final step is arranging the remaining $\ell - j$ elements into $m+1$ sections in $r_{m, \ell-j}$ ways. 
Summing $j$ over the range $0 \le j \le \el $ completes the proof.
\end{proof}


The following results express $r_{m,\el}$ in term of $s_{i,\el}$ for $0\le i\le m$.

\begin{lemma}
\label{kES}
The number of \BPA{}s on $ [\el] $ with $ m $ bars and exactly $ k $ empty sections is 
\[
	\ch{m + 1}{k} s_{ m - k, \el}
\]
for $ \el \ge 1 $. 
\end{lemma}

\begin{proof}

First, we must choose the $ k $ particular sections that will be empty. 
Then, we can have no other empty sections. So, we preferentially arrange $ [\el] $ into the remaining $ m - k + 1 $ sections, equivalent to $ m - k $ bars, so that we have no other empty sections. 
By definition, we can do this in $ s_{ m - k, \el } $ ways. Hence, there are 
\[
	\ch{m + 1}{k} s_{ m - k, \el}
\]
\BPA{}s with exactly $ k $ empty sections. 
\end{proof}

\begin{cor} For $m\ge 0$ and $ \el \ge 1 $, we have
\[
	r_{ m, \el} = \sum_{ k = 0}^{ m} \ch{ m + 1}{k + 1} s_{ k, \el }. 
\]
\end{cor}

\begin{proof} 
For $ \el \ge 1 $, any \BPA{} must have from 0 to $ m $ empty sections. 
From Theorem \ref{kES}, the number of \BPA{}s with $ k $ empty sections is 
$ \ch{m + 1}{k} s_{ m - k, \el} $. 
Summing over $ k $ gives
\[
	r_{ m, \el} = \sum_{ k = 0}^{ m} \ch{ m + 1}{k + 1} s_{ k, \el }. 
\]
\end{proof}


\section{Convergent and Asymptotic Series}

Gross \cite{gross:pas} showed that $r_\el$ is the sum of the infinite series
\begin{equation}
	r_\el = {1\over 2}\sum_{k\ge 0} {k^\el \over 2^k}.
	\label{eq:infsum}
\end{equation}
We can generalize this result to $r_{m,\el}$ as follows.

\begin{thm} 
For $m\ge 1$ and $\el\ge 0$, we have
\[
	r_{m, \el} 
	= \f{ 1}{ 2^{ m + 1} \,m!} \sum_{ k \ge 0 } {(k + 1)^{\overline{m}}\, k^{\el} \over 2^{k}},
\] 
where $x^{\overline{m}} = x(x+1)\cdots(x+m-1)$. 
\end{thm}

\begin{proof}

Substituting (\ref{eq:infsum}) in Theorem \ref{Main}, we have
\begin{align*}
	r_{ m, \el} 
	&= \f{ 1}{ 2^{ m} \,m!} \sum_{ 0\le i\le m} \cy{ m + 1}{ i + 1} r_{ \el + i} \\
	 &= \f{ 1}{ 2^{ m + 1} \, m!} \sum_{ 0\le i\le m} \cy{ m + 1}{ i + 1} 
	 \sum_{ k \ge 0 } {k^{\el + i} \over 2^{k}}.
\end{align*}
Interchanging the order of summation and using the identity
$\sum_{1\le r\le s} \cy{r}{s} x^s = x^{\overline{s}}$
(see \GKP{} \cite{graham:concrete}, p.~250), we obtain
\begin{align*}
	r_{m, \el} 
	&= \f{ 1}{ 2^{ m + 1} \, m!} \sum_{ k \ge 0 } {k^{\el -1} \over 2^{k}} 
	\sum_{0\le i\le m} \cy{ m + 1}{ i + 1} k^{ i + 1} \\
	&= \f{ 1}{ 2^{ m + 1} \, m! } \sum_{ k \ge 0 } {k^{\el - 1} \, k^{\overline{m + 1}} \over 2^{k}} \\
	& = \f{ 1}{ 2^{ m + 1} \, m!} \sum_{ k \ge 0 } {(k + 1)^{\overline{m}} \, k^{\el} \over 2^{k}}. 
\end{align*}
\end{proof}

Gross \cite{gross:pas} pointed out that (\ref{eq:infsum}) implies the asymptotic formula
\begin{equation}
	r_\el \sim {\el! \over 2(\log 2)^{\el+1}}.
	\label{eq:asymp}
\end{equation}
To see this, observe that replacing the sum in (\ref{eq:infsum}) by an integral 
$\int_0^\infty x^\el \, 2^{-x}\, dx$ introduces an error that is at most the total variation of the integrand.
Since the integrand is unimodal, rising from $0$ to a maximum of $(\el / e\log 2)^\el$ at
$x = \el/\log 2$, then decreasing to $0$, the total variation is just twice the maximum.
The integral is $\int_0^\infty x^\el \, 2^{-x}\, dx = \int_0^\infty y^\el \, e^{-y} \, dy / (\log 2)^{\el+1} = 
\el!/(\log 2)^{\el+1}$.
The error is at most $2(\el / e\log 2)^\el \sim \el!/(2\pi \el)^{1/2} \, (\log 2)^\el$
(because by Stirling's formula $\el! \sim (2\pi \el)^{1/2} \, \el^\el \, e^{-\el}$).
Together these results yield (\ref{eq:asymp}).

Combining (\ref{eq:asymp}) with Theorem \ref{Main}, we obtain
\[
	r_{m,\el} \sim {(\el+m)! \over 2^{m+1} \, m! \, (\log 2)^{\el+m+1}}
\]
as $\el\to\infty$ with $m\ge 0$ fixed (because the sum in Theorem \ref{Main} is
dominated by the term with $i=m$).

Gross \cite{gross:pas} observed that the exponential generating function (\ref{eq:cayleygen}) can be used to obtain more precise asymptotic information concerning $r_\el$.
His argument leads to the estimate 
\begin{equation}
	r_\el
	= {\el! \over 2(\log 2)^{\el+1}} + O\left( {\el\,\el! \over ((\log 2)^2 + 4\pi^2)^{(\el+1)/2}} \right),
	\label{eq:asymp1}
\end{equation}
in which the error term is exponentially smaller than the main term (whereas the error term
given by the argument following (\ref{eq:asymp}) is smaller only by a factor of $O(1/\el^{1/2})$).
His argument is as follows.
The function $r(z)$ is analytic except for simple poles at the points $\log 2 + 2\pi i k$ with $k\in\Z$.
Setting $t(z) = r(z) - 1/2(\log 2 - z)$ yields a function with the same poles as $r(z)$, except that the pole closest to the origin (at $z=\log 2$) has been cancelled. 
The term $r_\el$ (which is $\el!$ times the coefficient of $z^\el$ in $r(z)$) differs from 
$\el!/2(\log 2)^{\el+1}$ (which is $\el!$ times the coefficient of $z^\el$ in $1/2(\log 2 - z)$)
by $\el!$ times the coefficient of $z^el$ in $t(z)$.
By the residue theorem (see \WW{} \cite{whittaker:modern}, Chapter VI), this difference is $(\el!/2\pi i)\oint (t(z)/z^{\el+1})\,dz$, where the integral is taken counterclockwise around any contour that encircles the origin, but does not encircle any other singularity of the integrand.
The other singularities of the integrand closest to the origin are the simple poles at
$z=\log2 \pm 2\pi i$, which are at distance $\rho = ((\log 2)^2 + 4\pi^2)^{1/2}$ from the origin.
Taking the contour to be a circle centered at the origin with radius $\rho - 1/\el$, we find that 
$t(z) = O(\el)$ on the circle, so the integrand is $O(\el/\rho^{\el+1})$.
Since the length of the contour is $O(1)$, we conclude that the error term is 
$O(\el\,\el!/((\log 2)^2 + 4\pi^2)^{(\el+1)/2})$, which yields (\ref{eq:asymp1}).

It is clear that better and better asymptotic formulas may be obtained by canceling more and  more poles of $r(z)$ and integrating the result around larger and larger circles.
When this is done, it is seen that the contributions of the successive pairs of poles form a convergent series, and this naturally raises the question of whether $r_l$ can be expressed as the sum of an infinite series that is convergent and also asymptotic (so that the error committed by truncating the series after any term is bounded by a constant times the first neglected term).
That this is so is the substance of the following theorem. 
(The situation here is reminiscent of the asymptotic series for the partition function 
that was discovered by Hardy and Ramanujan \cite{hardy:asymp}.
Rademacher \cite{rad:1,rad:2} later showed that by slightly altering the terms of the series,
it could be made convergent as well as asymptotic.)

\begin{thm}
\label{Inf2}
For $\el\ge 1$, we have
\begin{equation}
	r_{\el} = \f{ \el!}{2} \sum_{ k \in \Z } {1\over z_k^{ \el + 1} },
	\label{eq:main1}
\end{equation}
where $ z_k = \log 2 + 2 \pi k i $. 
This can be rewritten in terms of real quantities as
\begin{equation}
	r_{ \el} = \f{ \el!}{ 2 (\log 2)^{ \el + 1} } 
	+\sum_{ k \ge 1 } { \el!  \over ((\log 2)^2 + 4 \pi^2 k^2)^{( \el + 1)/2 }} \, 
	T_{ \el + 1} \lp \f{ \log 2 }{ \sqrt{ \log^2 2 + 4 \pi^2 k^2 } } \rp,
	\label{eq:main2}
\end{equation}
where $ T_n(x) $ is the $n$-th Chebyshev polynomial, defined by $\cos(n\theta) = T_n(\cos\theta)$. 
\end{thm}

\begin{proof}
Suppose $ \el\ge 1 $. 
For $N\ge 0$, consider the rectangle $ A_N $ in the complex plane, where $ A_N $ has vertices $ \pm (2N + 1) \pm (2N + 1) \pi i $. Consider the contour integral
\[
	I_N = \oint_{A_N} \f{ r(z)}{ z^{ \el + 1 } } \, dz,
\]
where $ A_N $ is traversed counterclockwise. 

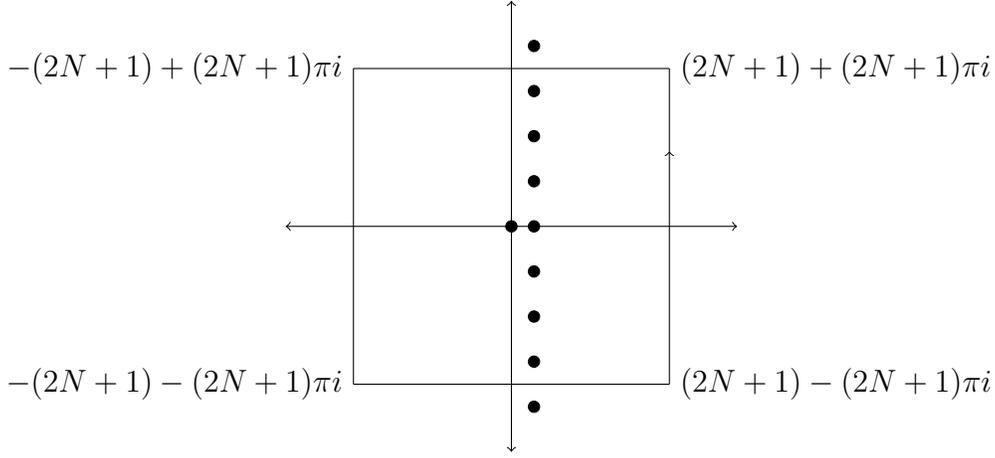
\begin{center}

\begin{tikzpicture}
\coordinate (1) at ( 0,0) {};
\coordinate (15) at (2.1,1.0) {};
\coordinate [label=right:$ (2N + 1) + (2N + 1) \pi i $] (2) at ( 2.1,2.1) {};
\coordinate [label=left:$ - (2N + 1) + (2N + 1) \pi i $] (3) at ( -2.1,2.1) {};
\coordinate [label=left:$ - (2N + 1) - (2N + 1) \pi i $] (4) at ( -2.1, -2.1) {};
\coordinate [label=right:$ (2N + 1) - (2N + 1) \pi i $] (5) at ( 2.1,-2.1) {};
\coordinate (6) at (0.3,0) {};
\coordinate (7) at (0.3, 0.6) {};
\coordinate (8) at (0.3, 1.2) {};
\coordinate (9) at (0.3, 2.4) {};
\coordinate (10) at (0.3, 1.8) {};
\coordinate (11) at (0.3, -0.6) {};
\coordinate (12) at (0.3, -1.2) {};
\coordinate (13) at (0.3, -1.8) {};
\coordinate (14) at (0.3, -2.4) {};
\draw (2) -- (3) ;
\draw (3) -- (4)  ;
\draw (4) -- (5) ;
\draw[->] (5) -- (15) ;
\draw (15) -- (2);
\draw[->] (1) -- (3,0);
\draw[->] (1) -- (0,3);
\draw[->] (1) -- (-3,0);
\draw[->] (1) -- (0,-3);
\draw [fill = black] (1) circle (0.075);
\draw [fill = black] (6) circle (0.075);
\draw [fill = black] (7) circle (0.075);
\draw [fill = black] (8) circle (0.075);
\draw [fill = black] (9) circle (0.075);
\draw [fill = black] (10) circle (0.075);
\draw [fill = black] (11) circle (0.075);
\draw [fill = black] (12) circle (0.075);
\draw [fill = black] (13) circle (0.075);
\draw [fill = black] (14) circle (0.075);

\end{tikzpicture}

\end{center}

\noi The integrand has singularities at $ z_k = \log 2 + 2 \pi k i $ for all $ k \in \Z $ and at $ 0 $. 
As $ N \to\infty$, $A_N $ eventually encloses all of these singularities. 
Thus, by the residue theorem,
\[
	\f{1}{ 2 \pi i } \lim_{ N \to \8} I_N = \left[ \Res \lp \f{ r(z)}{ z^{ \el + 1 } } ; 0 \rp + \sum_{ k \in \Z } \Res \lp \f{ r(z)}{ z^{ \el + 1 } } ; z_k \rp, \right]
\]
where $\Res(f(z),\zeta)$ denotes the residue of $f(z)$ at $z=\zeta$.
Since $|r(z)| \leq 1$ for  $z$ on $A_N$, the integrand is $O(1/N^{\el+1})$ on $A_N$.
Since $A_N$ has length $O(N)$, it follows that $I_N = O(1/N^\el)$.
Thus $I_N\to 0$ as $N\to\infty$,
so
\[
	\Res \lp \f{ r(z)}{ z^{ \el + 1 } } ; 0 \rp + \sum_{ k \in \Z } \Res \lp \f{ r(z)}{ z^{ \el + 1 } } ; z_k \rp = 0. 
\]
Since $\Res(r(z)/z^{\el+1}, 0) = r_\el/\el!$, we have
\[
	r_{\el} = - \el! \sum_{ k \in \Z } \Res \lp \f{ r(z)}{ z^{ \el + 1 } } ; z_k \rp. 
\]
Since $\Res(r(z)/z^{\el+1}, z_k) = -1/2z_k^{\el+1}$, we obtain (\ref{eq:main1}).

To obtain (\ref{eq:main2}), we note that except for $k=0$, the terms in (\ref{eq:main1}) come in pairs for which the real parts add and the imaginary parts cancel.
Thus
\begin{equation}
	r_\ell = \frac{\ell!}{2}\sum_{k \in \Z} {1\over z_k^{\ell+1}}
	= \frac{\ell!}{2(\log 2)^{\ell+1}}+ \ell! \sum_{k \ge 1} \Real \left(  {1\over z_k^{\ell+1}} \right).
	\label{eq:main3}
\end{equation}
We rewrite $1/z_k$ as $z_k = \rho_k e^{\theta_k}$, where $\rho_k = 1/((\log 2)^2 + 4\pi^2 k)^{1/2}$ and
$\cos \theta_k = \rho_k  \log 2$.
Since raising $1/z_k$ to the $(\el+1)$-st power raises its magnitude to the $(\el+1)$-st power and multiplies its angle by $\el+1$,
we obtain
\begin{align*}
	\Real\left({1\over z_	k^{\el+1}}\right) 
	&= \rho_k^{\el+1} \, \cos((\el+1)\theta_k) = \rho_k^{\el+1} \, T_{\el+1}(\cos \theta_k) \\
	&= {1\over ((\log 2)^2 + 4\pi^2 k)^{(\el+1)/2}} \, T_{\el+1}\left({\log 2 \over ((\log 2)^2 + 4\pi^2 k)^{1/2}}\right).
\end{align*}
Substituting this result into (\ref{eq:main3}) yields (\ref{eq:main2}).
\end{proof}

In principle we could apply the same method to find a convergent and asymptotic (for each fixed $m\ge 0$) series for
$r_{m,\el}$.
This, however, would require finding the residues of the higher-order poles of $r_m(z)$ at the points $z_k$, which is very awkward.
Instead, we combine the preceding theorem withTheorem \ref{Main} to obtain the following corollary.

\begin{cor}
For $m\ge 0$ and $\el\ge 1$, we have
\[
	r_{m,\el} = \f{ \el!}{ 2^{m+1} m!} \sum_{ i = 0 }^{m} 
	\cy{m + 1}{ i + 1}  \sum_{ k \in \Z } {1\over z_k^{ \el+i + 1} }.
\]
This can be rewritten in terms of real quantities as
\begin{align*}
	r_{ m,\el} &= \f{ 1}{ 2^m m!} \sum_{ i = 0 }^{m} \cy{m + 1}{ i + 1} \times \\
	&\qquad 
	\left( { (\el+i)! \over  
	2 (\log 2)^{ \el+i + 1} } 
	+\sum_{ k \ge 1 } { (\el+i)!  \over ((\log 2)^2 + 4 \pi^2 k^2)^{( \el +i+ 1)/2 }} \, 
	T_{ \el +i+ 1} \lp \f{ \log 2 }{ \sqrt{ \log^2 2 + 4 \pi^2 k^2 } } \rp\right). \\
\end{align*}
\end{cor}

\section{Acknowledgment}

The research reported in this paper was supported in part by grant CCF 0917026 from the National Science Foundation.

\end{document}